\documentclass[a4paper,11pt]{article}
\usepackage{amsmath,amssymb,amsfonts,mathrsfs,hyperref,microtype, amsthm}
\usepackage{eufrak}

\usepackage{color}

\newtheorem{Defi}{Definition}[section]
\newtheorem{thm}{Theorem}[section]
\newtheorem{prop}{Proposition}[section]
\newtheorem{cor}{Corollary}[section]
\newtheorem{rmk}{Remark}[section]
\newtheorem{lma}{Lemma}[section]

\newcommand{\KK}{\mathrm{K}}

\newcommand{\HH}{\mathcal{H}}
\newcommand{\Var}{\textbf{Var}}
\newcommand{\Z}{\mathbb{Z}}
\newcommand{\nn}{\mathbf{n}^\uparrow}
\def\N{{\rm I\kern-0.16em N}}
\def\R{{\rm I\kern-0.16em R}}
\def\E{{\rm I\kern-0.16em E}}
\def\P{{\rm I\kern-0.16em P}}
\def\F{{\rm I\kern-0.16em F}}
\def\B{{\rm I\kern-0.16em B}}
\def\C{{\rm I\kern-0.46em C}}
\def\G{{\rm I\kern-0.50em G}}

\numberwithin{equation}{section}
\font\eka=cmex10
\usepackage{ae}
\def\ind{\mathrel{\hbox{\rlap{%
\hbox to 7.5pt{\hrulefill}}\raise6.6pt\hbox{\eka\char'167}}}}
\begin{document}

\title{\textbf{The law of iterated logarithm for subordinated  Gaussian sequences: \\ uniform Wasserstein bounds}}

\author{Ehsan Azmoodeh, Giovanni Peccati and Guillaume Poly}
\maketitle

\abstract We develop a new method for showing that a given sequence of random variables verifies an appropriate law of the iterated logarithm. Our tools involve the use of general estimates on 
multidimensional Wasserstein distances, that are in turn based on recently developed inequalities involving Stein matrices and transport distances. Our main application consists in the proof of the exact 
law of the iterated logarithm for the Hermite variations of a fractional Brownian motion in the critical case.

\



\tableofcontents

\section{Introduction}

\subsection{Overview}

The aim of the present paper is to develop a new technique for proving laws of the iterated logarithm (LIL) for general sequences of random variables, possibly having the form of partial sums of random 
elements displaying some strong form of dependence. One of the main contributions of our work consists in a collection of sufficient conditions for the LIL to hold, expressed either in terms of uniform 
controls on the (multidimensional) Wasserstein distance between the elements of the sequence and some Gaussian counterpart, or in terms of some underlying collection of {\it Stein matrices} 
(see Definition \ref{d:sm} below). Stein matrices can be roughly described as arrays of random variables verifying a generalised integration by parts formula: they appear naturally when implementing 
the so-called {\it Stein's method} for normal approximations, see \cite{n-pe-book} for an introduction to this topic. One of the key technical tools developed in our work is the new inequality 
\eqref{comparison}, that we believe has a remarkable independent interest, providing an explicit bound on the multidimensional Kolmogorov distance in terms of the 1-Wasserstein distance, where the 
involved constants display a logarithmic dependence in the dimension. In the proof of our main estimates, we shall often make use of the recent findings from \cite{HSI, Entropy}, where a new connection 
between Stein matrices and information functionals has been revealed, thus yielding new bounds on transport distances. 

In what follows, every random element is defined on a common probability space $(\Omega, \mathcal{F}, \P)$.

\subsection{Motivation: fractional Hermite variations in the critical regime}

Let $B^H = \{B_t : t\in \R\}$ be a standard fractional Brownian motion on the real line with Hurst parameter $H\in (0,1)$, that is: $B^H$ is a centered Gaussian process having covariance 
$\E[B_s^HB_t^H] = 2^{-1}[ |t|^{2H}+|s|^{2H} - |s-t|^{2H}]$. Write $Z^H_k := B^H_{k+1}- B^H_k$, $k\in \Z$, and denote by $\{H_q : q=0,1,...\}$ the usual collection of Hermite polynomials 
(so that $H_0=1$, $H_1(x) = x$, $H_2(x) = x^2-1$, and so on; see e.g. \cite[Section 1.4]{n-pe-book}). We are interested in the asymptotic behavior of the so-called {\it Hermite variations} of $B^H$, that is, 
we want to study random sequences of the type $n \mapsto V_n:= \sum_{k=1}^n H_q(Z^H_k)$, as $n\to\infty$, for fixed values of $q$ and of the Hurst index $H$.
It is a well-known fact that the fluctuations of such variations heavily depend on the relation between $q$ and $H$, a crucial role being played by the so-called `critical regime', corresponding to the 
choice of parameters $H = 1-\frac{1}{2q}$. The following convergence results, involving two well-known Central Limit Theorems (CLTs), are classical:
\begin{itemize}
\item[(a)] ({\it Breuer-Major CLT}, see e.g. \cite[Chapter 6]{n-pe-book}) If $H\in (0, 1-1/2q)$, then there exists a finite constant $\sigma_q>0$, such that the sequence $n^{-1/2}V_n$ converges in 
distribution to a centered Gaussian random variable with variance $\sigma_q^2$.

\item[(b)] ({\it Non-central convergence}, see \cite{DM79, Taqqu3}) If $H> 1-1/2q$, then the sequence $n^{q(1-H)-\frac12}V_n$ converges in distribution to a non-Gaussian random variable, having a so-called `Hermite distribution'.

\item[(c)] ({\it CLT in the critical regime}, see e.g. \cite{GS85}) If $H= 1-1/2q$, then, for some appropriate constant $\sigma_q>0$,  $(\log n)^{-1/2}V_n$ converges in distribution to a centered Gaussian random variable with variance $\sigma_q^2$.

\end{itemize}

The reader is referred to \cite[Section 7.4]{n-pe-book} for a unified modern presentation of these phenomena. The following question is therefore natural: {\it can one associate an exact law of the iterated logarithm (LIL) to each one of the convergence results described at Points {\rm(a)}, {\rm(b)} and {\rm(c)}}? It turns out that, although an appropriate LIL has been shown in the two cases (a) and (b) (see the discussion below), none of the available techniques can be used to deal with the critical case (c). It will be demonstrated that our new approach exactly allows to fill this { fundamental} gap. 

We will now provide a discussion of the available results concerning LILs for subordinated Gaussian sequences.

\medskip

\noindent{\it Case {\rm (a)}}. Let $Z = \{Z_k : k\in \Z\}$ be a centered stationary Gaussian sequence, and let $f$ be a measurable and square-integrable mapping. Since the seminal results by Breuer and Major (see \cite{BM83}, as well as \cite[Chapter 7]{n-pe-book}), many authors tried to deduce criteria on $f$ and $Z$ ensuring that, for some adequate finite constant $\sigma>0$,
\begin{equation}\label{e:haden}
\limsup_{n\to\infty} \frac{1}{\sqrt{2n\log\log n}}\sum_{k=1}^n f(Z_k)=\sigma,
\end{equation}
with probability one. Relying on a seminal paper of Lai and Stout \cite{LS1980} which provides conditions for the upper-bound of the iterated logarithm for general partial sums of dependent random variables, and by using systematically the so-called `method of moments', Arcones \cite{arcones} and Ho \cite{Ho} obtained LILs for non-linear functionals of general Gaussian fields. First, Ho \cite{Ho} { has provided criteria ensuring that the right-hand side of \eqref{e:haden} is bounded from above by some finite constant $\sigma$}, by expressing the conditions of Lai and Stout in terms of the covariance of $Z$ and the coefficients of the Hermite expansion of $f$. Next, Arcones \cite{arcones} has extended the results of Ho, in particular by  obtaining exact lower bounds. The key idea developed by Arcones in order to obtain lower bounds, is to consider Gaussian stationary sequences of the form
\begin{equation}\label{Repfbm}
G_k=\sum_{n=\infty}^{\infty} a_{n+k} N_n,
\end{equation}
and next to use the classical law of the iterated logarithm for locally dependent sequences by a truncation argument. It turns out that some of the results by Arcones contain the exact law of LIL associated with the CLT at Point (a).  Indeed, whereas it is not obvious at first glance, one can represent the increments of the $B^H$ in the form (\ref{Repfbm}) (see for instance \cite{NuTyndel}). Besides, the coefficients in the expansion (\ref{Repfbm}) are such that $a_k\sim\frac{1}{k^{\frac 3 2 -H}}$ (see \cite[prop. 2.2, p.64]{Beran}). Plugging these facts in \cite[Proposition 1]{arcones}, one deduces immediately that, if $H<1-\frac{1}{2q}$, then
$$\limsup_{n\to\infty}\frac{1}{\sqrt{2n\log\log(n)}}\sum_{k=1}^n H_q(B_{k+1}^H-B_k^H)=\sigma_q>0,$$
with probability one.

\medskip

\noindent{\it Case {\rm (b)}}.  The question of the iterated logarithm in this setup was partially solved by Taqqu in \cite{Taqqu}. Later on, Lai and Stout \cite{LS1980} gave criteria for upper bounds, whereas the complete law of the iterated logarithm was proved by Mori and Oodaira in \cite{Moda}. 

\medskip

\noindent{\it Case {\rm (c)}}. The first LIL ever proved for the critical regime (c) will appear in Theorem \ref{exam-critical} below: the proof is based on the novel approach developed in the present work. Note that, so far, there has been no attempt to prove a LIL in this delicate context. We believe indeed that it would be not possible (or, at least, technically very demanding) to extend the approaches by Arcones \cite{arcones} and Mori and Oodaira in \cite{Moda} to deal with this case. One plausible explanation for this impasse is that, in both cases (a) and (b), the convergence in distribution takes place at an algebraic speed in $n$ (with respect e.g. to the Kolmogorov distance, see e.g. \cite[p. 146]{n-pe-book}). However, it is known since \cite[p. 381]{Bi11} that the speed of convergence is logarithmic in the critical regime (c), and such a rate is sharp. A careful analysis  of the proofs of Mori Oodaira and Arcones reveals that most arguments in their approach are based on `polynomial' estimates in the truncations, derived from upper bounds on moment sequences: as they are, such estimates are of no { use} for dealing with a logarithmic speed of convergence. In contrast, our approach allows one to obtain a simple and transparent proof of the LIL stated in Theorem \ref{exam-critical}, thus by-passing at once the difficulties mentioned above.

\subsection{Stationary Gaussian sequences}
 
 As a by-product of our analysis, in Theorem Theorem \ref{ILLGaussian} we shall obtain a very general LIL for a stationary Gaussian sequence $Z$. To our knowledge, the most general LIL for a stationary Gaussian sequence is due again to Arcones \cite{arconesGaussian}. In such a reference, the author shows that the LIL holds under the condition that $\sum_{k} |\rho(k)|<\infty$, where $\rho$ is the correlation function of $Z$ (this covers the result of Deo \cite{DEO74}). Other conditions were given in \cite{LS1978,Taqqu} which are similar to the condition we provide in the Theorem \ref{ILLGaussian}, in the sense that it is required that the variance of the sequence of partial sums is asymptotically equivalent to a sequence of the type $n^\alpha L(n)$, where $L$ is a regularly varying function. We stress that there is an important difference between our work and some of the existing literature, namely: we do not need any further assumptions on the function $L$, whereas both references \cite{LS1978,Taqqu} need some additional technical requirements on $L$. Finally, we stress that our condition covers the findings of \cite{arconesGaussian}, see Corollary \ref{CoroGaussian} below. Our findings support the conjecture that the law of the iterated logarithm in this setting holds under \textit{the only} assumption that the variance is regularly varying (meaning that Assumption \ref{covarianceassup} below can be dropped). 

\subsection{Remark on notation}

Throughout the paper, we shall use standard notations from Malliavin calculus -- the reader is referred to \cite[Chapters 1 and 2]{n-pe-book} for a standard introduction to this topic. In particular, given an isonormal Gaussian process $G=\{G(h) : h\in \mathcal{H}\}$ over some real separable Hilbert space $\mathcal{H}$, we shall denote by $D$ and $\delta$, respectively, the {\it Malliavin derivative} and {\it divergence} operators. Also, we shall write $L$ to indicate the {\it generator} of the associated Ornstein-Uhlenbeck semigroup. We recall that a square-integrable functional $F$ of $G$ is said to belong to the {\it $q$th Wiener chaos} associated with $G$ (for $q=0,1,2,...$) if $LF = -qF$. { We also recall, for future use, the following crucial 
{\it hypercontractivity property} of Wiener chaoses (see e.g. \cite[Corollary 2.8.14]{n-pe-book} for a proof): if $F$ is an element of the $q$th Wiener chaos of a given Gaussian field, then, for every $r>p>1$,
\begin{equation}\label{e:hc}
\E[|F|^r]^{1/r} \leq \left(\frac{r-1}{p-1}\right)^{q/2} \!\!\times\,\, \E[|F|^p]^{1/p}.
\end{equation}
}
\subsection{Plan}

The paper is organized as follows. Section 2 contains the statements of our main results. Section 3 is devoted to some preliminary material, whereas Section 4 and 5 contain the proofs, respectively, of our theoretical results and of our findings connected to applications.

\section{Statement of the main results}

Throughout the present section, we will consider a sequence $$X= \{X_n: n\geq 0\}$$ of real-valued random variables that are defined on a common probability space $(\Omega,\mathcal{F},\P)$. We make the convention that $X_0=0$, and we assume that the elements of the sequence $X$ are 
centered, i.e., that $\E[X_n]=0$ for all $n\ge1$. In general, the capital letter $C$ stands for a general
constant which may vary from line to line; its dependency on other parameters at hand will be emphasized whenever it is important. 

\medskip

Given two random elements ${\bf Z}, {\bf Y}$ with values in $\R^d$ ($d\geq 1$), the {\em Kolmogorov distance} between the laws of ${\bf Z}$ and ${\bf Y}$, denoted $d_{\rm K}({\bf Z}, {\bf Y})$, is defined as follows:
$$
d_{\rm K}({\bf Z}, {\bf Y}) = \sup\Big| \P[{\bf Z} \in Q] - \P[{\bf Y} \in Q] \Big|,
$$
where the supremum runs over all rectangles of the form $Q = (-\infty , a_1]\times \cdots \times (-\infty , a_d]$, with $a_1,...,a_d\in \R$.
\medskip

Fix $\theta\geq 1$. Given two random elements ${\bf Z}, {\bf Y}$ with values in $\R^d$ ($d\geq 1$) and such that $\E\|{\bf Z}\|_{\R^d}^\theta, \, \E\|{\bf Y}\|_{\R^d}^\theta<\infty$, the {\em $\theta$-Wasserstein distance} $W_\theta({\bf Z}, {\bf Y})$ between the laws of ${\bf Z}$ and ${\bf Y}$ is given by 
$$
W_\theta({\bf Z}, {\bf Y}) := \inf\left \{  (\E[\|{\bf U}-{\bf V} \|_{\R^d}^\theta])^{1/\theta}\right\},
$$
where the infimum runs over all $2d$-dimensional vectors $({\bf U}, {\bf V})$ such that ${\bf U}\stackrel{law}{=}{\bf  Z}$ and ${\bf V}\stackrel{law}{=} {\bf Y}$. The value of the dimensional parameter $d$, which does not appear in the notation $W_\theta({\bf Z}, {\bf Y})$, will be always clear from the context.

Given two sequences of real numbers $(u_n)_{n\ge 1}$ and $(v_n)_{n\ge 1}$, the notation $u_n\sim v_n$ means that
$$\lim_{n\to\infty} \frac{u_n}{v_n}=1.$$

\subsection{A general law of the iterated logarithm}\label{ss:genLIL}

We shall now introduce a collection of assumptions, that will enter the statements of our main results.

\begin{itemize}
\item[\textbf{(A1)}] The sequence $X$ verifies Assumption {\bf (A1)} if there exists a slowly varying function $L$ ({ that is, $L$ is such that $\lim_{x\to\infty} L(ax)/L(x)=1$ for every $a>0$ --- see e.g. \cite[p. 14]{Bingbook}}) and a function $g:\mathbb{N} \to \R_+$ such that for some $a\in]0,1]$

$$g(n)\sim n^a L(n),$$
and for some $C>0$ and all $n_1<n_2$
\begin{equation}\label{Variance_r}
\left|\E\Big[  \frac{ X_{n_2}-X_{n_1}}{g(n_2-n_1)} \Big]^2-1\right|\le \frac{C}{1+\log(n_2-n_1)}.
\end{equation}

\item[\bf (A2)] We shall say the $X$ verifies Assumption {\bf (A2)} if, for every pair of integers $n,p\geq 1$,
$$
\limsup_{a\to \infty} \E[(X_{a+n} -X_a)^{2p}] <\infty. 
$$
\item[\bf (A3)]  Let $G$ be a one-dimensional standard Gaussian random variable, let $X$ verify assumption {\bf (A1)}, and let $g : \R_+\to \R_+$ be the corresponding function. We say that $X$ verifies Assumption {\bf (A3)} if there exist constants $C,\lambda>0$ such that, for all $\theta \geq 1$, 
$$
W_{\theta}\left(\frac{X_{n_2}-X_{n_1}}{g(n_2-n_1)}, G\right) \leq C\alpha(\theta) \frac{\theta^\lambda}{1+\log(n_2-n_1)},
$$
where $\alpha(1) = 1$ and $\alpha(\theta) = (\theta-1)^\frac12$ for $\theta>1$, and moreover
$$
d_\KK\left(\frac{X_{n_2}-X_{n_1}}{g(n_2-n_1)}, G\right)\leq \frac{C}{1+\log(n_2-n_1)}
$$
for every $n_2>n_1$.

\item[\bf (A4)]  

Given real numbers $q>1$, $\alpha>0$, and  integers $d,m\ge 1$, we consider the particular collection of positive integers (with $[x]$ the integer part of the real number $x$)
$$\nn:=\{n_i\}_{1 \le i \le 2d} =\big\{[q^{{(m+i)}^{1+\alpha}}]\big\}_{1 \le i\le 2d}.$$
For simplicity we write
\begin{eqnarray}\label{e:freccia}
\mathbf{Y}_{\nn}&=&(Y_{1},\cdots,Y_{d})\\ &=&\left(\frac{X_{n_2}-X_{n_1}}{g(n_2-n_1)},\frac{X_{n_4}-X_{n_3}}{g(n_4-n_3)},\cdots,\frac{X_{n_{2d}}-X_{n_{2d-1}}}{g(n_{2d}-n_{2d-1})}\right),\notag
\end{eqnarray}
the random vector or size $d$ of increments of X along the subsequence $\nn$. We say that the sequence $X$ verifies Assumption {\bf (A4)} if, { for some fixed $q>1$ and every $\alpha>0$, there exists some constant $C_{\alpha,q}$ such that, for every $d,m\ge 1$, }
\begin{equation}\label{e:multiwass}
W_1(\mathbf{Y}_{\nn}, {\bf G}) \leq \frac{d \, C_{\alpha,q}}{1+\log(n_{2}-n_{1})},
\end{equation} 
where ${\bf G}$  stands for a $d$-dimensional vector of i.i.d. centered standard Gaussian random variables.


%
%
%
 
%

\end{itemize}
\begin{rmk}{\rm
Roughly speaking, assumption (\textbf{A4}) expresses the fact that the normalized increments of X taken at the particular scale $q^{i^{1+\alpha}}$,  behave as independent Gaussian. Moreover, the error in this approximation for the Wasserstein distance is logarithmic in the size of the smallest increment $(n_2-n_1)$.}
\end{rmk}

The next statement is one of the main achievements of the present paper.

\begin{thm}\label{mainthm0}
Assume that the sequence $\{X_n : n\geq 1\}$ satisfies the four assumptions {\rm (\textbf{A1})--(\textbf{A4})}. Then,

\begin{equation}
 \limsup_{n\to\infty} \frac{X_n}{\sqrt{2  g^2(n) \log \log n}}= 1, \quad \text{a.s.}
\end{equation}

\begin{equation}
\liminf_{n\to\infty} \frac{X_n}{\sqrt{2 g^2(n) \log \log n}}= -1, \quad \text{a.s.}
\end{equation}
where the mapping $g$ appears in Assumption {\rm (\textbf{A1})}.

\end{thm}

\begin{rmk}{\rm
As demonstrated below, the Assumption \textbf{(A3)} expresses a sort of hypercontractivity. However, an inspection of the proof of the Theorem \ref{mainthm0} reveals that the mere Assumptions \textbf{(A1)}, \textbf{(A2)} and \textbf{(A4)} are enough to ensure that
$$\limsup_{n\to\infty}\frac{X_{k_n}}{g(k_n)\sqrt{2\log\log(k_n)}}=1\,\,\,\,\text{a.s.},$$
where $k_n=q^{n^{1+\alpha}}$.}
\end{rmk}

\subsection{Checking the assumptions by means of Stein matrices}
We will now show how one can check the validity of Assumptions {\bf (A2)--(A4)} of the previous section by using the concept of a {\it Stein matrix} associated with a given random vector. As discussed below, such a notion is particularly well adapted for dealing with the normal approximation of functionals of general Gaussian fields.

\begin{Defi}[Stein matrices]\label{d:sm} {\rm Fix $d\geq 1$, let ${\bf F} = (F_1,...,F_d)$ be a $d$-dimensional centered random vector, and denote by $M(d,\R)$ the space of $d\times d$ real matrices. We say that the matrix-valued mapping
$$
\tau : \R^d \to M(d, \R) : {\bf x} \mapsto \tau({\bf x}) = \{ \tau_{i,j}({ \bf x}) : i,j=1,...,d \}
$$
is a {\it Stein matrix} for ${\bf F}$  if $\tau_{i,j} (F)\in L^1(\P)$ for
every $i,j$ and the following is verified: for every
differentiable function $g : \R^d \to \R$ such that $g$ and its partial derivatives have at most polynomial growth, the two (vector-valued) expectations $\E \left[ {\bf F} g({\bf F})  \right]$ and $\E \left[\tau({\bf F})  \nabla g({\bf F})  \right]$ are well defined and 
\begin{equation}
\label{eq:26}
\E \left[ {\bf F} g({\bf F})  \right] = \E \left[\tau({\bf F})  \nabla g({\bf F})  \right],
\end{equation}
or, equivalently,
\begin{equation}
\label{eq:26bis}
\E \left[ F_i g({\bf F}) \right] = \sum_{j=1}^d \E \left[ \tau_{i,j}({\bf F}) \partial_jg({\bf F}) \right], \quad i=1,...,d.
\end{equation}
Note that, selecting $g(x) = x_j$, $j=1,...,d$, one obtains from \eqref{eq:26} that $\E[F_iF_j] = \E[\tau_{i,j }(F)]= \E[\tau_{j,i }(F)]$, for every $i,j=1,...,d$. Finally, we stress that, in dimension $d=1$ the Stein matrix $\tau$ is simply a real-valued mapping, which is customarily called a {\it Stein factor}.}
\end{Defi}

The next statement provides an explicit connection between properties of Stein matrices and the law of the iterated logarithm stated in the previous section.

\begin{prop}\label{p:checksm} Let $X= \{X_n : n\geq 1\}$ be the sequence of centered random variables introduced in the previous section. Assume that $X$ verifies Assumption {\bf (A1)} (for some adequate mapping $g$), and also that the following properties hold: 
\begin{itemize}
\item[\rm (i)] For each $d\ge 1$, and each increasing sequence $\nn=\{n_i\}_{1\le i \le 2 d}$ of $2d$ integers, the vector $\mathbf{Y}_{\nn}=(Y_{1},\cdots,Y_{d})$, as defined in \eqref{e:freccia}, admits a $d\times d $ Stein matrix $\boldsymbol{\tau}_{\nn}=\{\tau_{i,j} : i,j=1,...,d\}$, in the sense of Definition \ref{d:sm}.

\item[\rm (ii)] There exists $q>1$ such that, for every $\alpha>0$, there exists a constant $C_{\alpha,q}>0$ verifying the inequalities
\begin{equation}\label{varGamma1}
\sqrt{\Var\Big[\tau_{i,i}(\mathbf{Y}_{\nn})\Big]}\le\frac{C_{\alpha,q}}{1+\log(n_{2i}-n_{2i-1})},\,\forall i=1,...,d,
\end{equation}
and {
\begin{equation}\label{varcrossGamma1}
A(i,j)\le \frac{C_{\alpha,q}}{1+\log(n_{2i}-n_{2i-1})}, \,\forall 1\leq i<j\leq d,
\end{equation}
where $$A(i,j) = \max\left\{ \sqrt{\E\Big[\tau_{i,j}(\mathbf{Y}_{\nn})^2\Big]} , \sqrt{\E\Big[\tau_{j,i}(\mathbf{Y}_{\nn})^2\Big]}\right\},$$ }
for every $d\geq 1$ and every increasing collection of integers of the type $\nn=\{n_i\}_{1 \le i \le 2d} =\big\{[q^{{(m+i)}^{1+\alpha}}]\big\}_{1 \le i\le 2d}$, where $m\geq 1$. Here, we have adopted the notation \eqref{e:freccia}, whereas ${\bf \tau}_{\nn} = \{\tau_{i,j}\}$ is the Stein matrix associated with $\mathbf{Y}_{\nn}$. 

\item[\rm (iii)] There exist constants $C,\lambda >0$ such that, for all $\theta\geq 1$,

\begin{eqnarray}\notag
\Big\Vert  \tau\left(\frac{X_{n_2}-X_{n_1}}{g(n_2-n_1)}\right) -1 \Big \Vert_{\theta}&:=& \Big( \E \big\vert \tau\left(\frac{X_{n_2}-X_{n_1}}{g(n_2-n_1)}\right) -1 \big\vert^\theta \Big)^{\frac{1}{\theta}}\\
& \le& C \frac{\theta^\lambda}{1+\log {(n_2-n_1)}}, \label{thetahypercontract}
\end{eqnarray}
for every $n_1<n_2$, where $\tau$ stands for the Stein factor of $\frac{X_{n_2}-X_{n_1}}{g(n_2-n_1)}$.
\end{itemize}
Then, $X$ verifies assumptions {\rm {\bf (A2)}, {\bf (A3)}} and {\rm {\bf (A4)}}.
\end{prop}

\begin{rmk}\label{remark-assumption(B)} {\rm
If the random sequence $X =\{X_n: n\geq 1\}$ is composed of functionals of an isonormal Gaussian process $G= \{G(h) : h\in \mathcal{H}\}$ and if each $X_n$ lies in the domain of the Malliavin derivative operator $D$ , then the previous assumption (i) is always fulfilled by taking 
$$\tau_{i,j}(Y_1,\cdots,Y_d)=\E\Big[\langle D Y_j, -D L^{-1} Y_i \rangle_\HH\left|\right. (Y_1,\cdots,Y_d)\Big],$$
where $L^{-1}$ stands for the pseudo-inverse of the Ornstein-Uhlenbeck generator { (see e.g. \cite[Section 2.8.2]{n-pe-book})}. In particular, if the sequence $X$ belongs to the $q$th Wiener chaos of $G$ (and therefore $L^{-1} X_n = -q^{-1} X_n$ for every $n$), one has the simple representation
\begin{equation}\label{e:simplesm}
\tau_{i,j}(Y_1,\cdots,Y_d)=\frac1q \E\Big[\langle D Y_j, D Y_i \rangle_\HH\left|\right. (Y_1,\cdots,Y_d)\Big],
\end{equation}
which also implies that the Stein matrix $\{\tau_{i,j}(Y_1,...,Y_d) : i,j=1,...,d\}$ is symmetric. Again, we refer the reader e.g. to \cite{Entropy} for a concise exposition of the required notions and to the monographs \cite{n-pe-book, NUbook} for more details.}
\end{rmk}

\begin{rmk}\label{D is empty}{\rm
When $\{X_n : n\geq 0\}$ lies in a finite sum of Wiener chaoses, Assumption (iii) is particularly easy to check. Indeed, using hypercontractivity properties $(\ref{e:hc})$, it is sufficient to check equation (\ref{thetahypercontract}) only in the case $\theta=2$. This case is indeed covered by Assumption (ii).}
\end{rmk}

{
\subsection{First examples: LIL for independent sequences}

As demonstrated in the sections to follow, the techniques developed in the present paper have been specifically devised for deducing laws of the iterated logarithm involving sums of random variables displaying some form of non-trivial dependence. However, in order to develop some intuition about the assumptions appearing in the statements of Theorem \ref{mainthm0} and Proposition \ref{p:checksm}, it is instructive to first focus on the case of independent random variables. We stress that the aim of this section is to provide an illustration of our techniques in a familiar framework: in particular, we do not aim at generality. The reader is referred e.g. to \cite[Section 12.5]{Dudley} for an exhaustive discussion of the LIL (and its history) for sequences of i.i.d. random variables.

\subsubsection{Rademacher sequences}

We start by considering the case of independent Rademacher random variables $\{\varepsilon_i : i\geq 1\}$ (that is, $\P[\varepsilon_i = 1] = 1/2 = \P[\varepsilon_i = -1]$, $i\geq 1$). In this case, it is well known that, by noting $S_n = \sum_{i=1}^n \varepsilon_i$, $n\geq 1$,
\begin{eqnarray}\label{e:ll}
&& \limsup_{n\to\infty} \frac{S_n}{\sqrt{2  n \log \log n}}= 1=- \liminf_{n\to\infty} \frac{S_n}{\sqrt{2 n \log \log n}},\end{eqnarray}
with probability one. In what follows, we shall show that \eqref{e:ll} can be directly deduced from Theorem \ref{mainthm0} and Proposition \ref{p:checksm}. In order to accomplish this task, it is indeed preferable to show the equivalent statement: with probability one,
\begin{eqnarray}\label{e:lll}
&& \limsup_{n\to\infty} \frac{X_n}{\sqrt{2  n \log \log n}}= 1 =- \liminf_{n\to\infty} \frac{X_n}{\sqrt{2 n \log \log n}},\end{eqnarray}
where $X_n = S_n+U$ and $U$ is a random variable uniformly distributed on $[-1,1]$, independent of the $\varepsilon_i$'s. It is clear that, in this case, Assumption {\bf (A1)} is verified for the choice of function $g(n) = \sqrt{n}$. Moreover, a simple computation (based e.g. on \cite[Lemma 3.3]{cha}), shows that, for any choice of $\nn$, the vector $\mathbf{Y}_{\nn}$ appearing at Assumption (i) of Proposition \ref{p:checksm} admits a Stein matrix ${\bf \tau}_{\nn}= \{\tau_{i,j} : i,j=1,...,d\}$ such that, for $i=1,...,d$,
$$
\tau_{i,i} =\frac{\E\left\{ n_{2i} - n_{2i-1} -(S_{n_{2i}} - S_{n_{2i-1}}) U +(1-U^2)/2 \,\Big | \, \mathbf{Y}_{\nn} \right\}}{n_{2i} - n_{2i-1}},
$$
and, for $1\leq i\neq j\leq d$,
$$
\tau_{i,j} =\frac{1}{\sqrt{(n_{2i} - n_{2i-1})(n_{2j} - n_{2j-1})}} \E\left\{ (1-U^2)/2 \,\Big | \, \mathbf{Y}_{\nn} \right\},
$$
and it is a matter of a simple verification to check that Assumption (ii) of Proposition \ref{p:checksm} is indeed satisfied. Finally, since for any choice of integers $n_2>n_1$ one has that
$$
\tau\left( \frac{X_{n_2}-X_{n_1}}{\sqrt{n_2-n_1}}\right) \!\!=\!\!\frac{\E\left\{ n_{2} - n_{1}\! -\!(S_{n_{2}} - S_{n_{1}}) U \!+\!(1-U^2)/2 \,\Big | \,  \frac{X_{n_2}-X_{n_1}}{\sqrt{n_2-n_1}}\right\}}{n_{2} - n_{1}},$$
we deduce immediately from a standard application of Khinchin inequality that, for every $\theta >2$,
$$
\Big\Vert  \tau\left(\frac{X_{n_2}-X_{n_1}}{\sqrt{n_2-n_1}}\right) -1 \Big \Vert_{\theta}\leq \frac{(\theta-1)^{1/2} +1}{n_2-n_1}.
$$
This implies in particular that Assumption (iii) in Proposition \ref{p:checksm} is verified, and consequently that \eqref{e:lll} (and therefore \eqref{e:ll}) holds.

\medskip

The crucial point of the previous example is of course that, although the random variables $S_n$ are discrete and not directly amenable to analysis by means of our techniques, the simple addition of the independent bounded component $U$ makes Stein matrices appear very naturally. For the time being, it is unclear whether a similar smoothing operation can be realised for an arbitrary sequence of independent discrete random variables.

\subsubsection{Random variables with densities}

We now consider a sequence $\{Z_ i: i\geq 1\}$ of i.i.d. real-valued centered random variables with unit variance. We assume that the law of $Z_1$ is absolutely continuous with respect to the Lebesgue 
measure, with a density $f$ whose support is assumed (for simplicity) to be a (possibly unbounded) interval. Writing $X_n = S_n = \sum_{i=1}^nZ_i$, it is of course well-known 
(see e.g. \cite[Theorem 12.5.1]{Dudley}) that relation \eqref{e:lll} holds with probability one. In what follows, we shall show that, under some additional assumption on the regularity of the density $f$, 
such a result can be directly deduced from Theorem \ref{mainthm0} and Proposition \ref{p:checksm}. To this end, we define an auxiliary function $s : \R\to \R$ as follows:
$$
s(x) = \frac{\int_x^\infty yf(y)dy}{f(x)},
$$
for each $x$ in the support of $f$, and $s(x) = 0$ otherwise. Then, it is a simple exercise in integration to show that the random variable $s(Z_1)$ is indeed a Stein factor for $Z_1$, that is: for every 
smooth mapping $\varphi$, one has that $\E[Z_1\varphi(Z_1)] = \E[s(Z_1)\varphi'(Z_1)]$. Moreover, one can check that $s(Z_1)\geq 0$ with probability one (see e.g. \cite[Chapter VI]{S86}). The following 
statement contains the announced connection with the main results of the present paper.

\begin{prop}\label{p:trio} Let the above notation and assumptions prevail. If $\E[s(Z_1)^2]<\infty$, then Assumption {\bf (A1)} together with Assumptions {\rm (i)} and {\rm (ii)} of Proposition 
\ref{p:checksm} are verified for the choice of function $g(n) = \sqrt{n}$. If moreover there exist constants $0<\lambda, K<\infty$ such that 
\begin{equation}\label{e:cd}
 \E[s(Z_1)^\theta] ^{1/\theta} \leq  K\theta^\lambda, \quad \theta\geq 2,
\end{equation}
then also Assumption {\rm (iii)} of Proposition \ref{p:checksm} is satisfied.
\end{prop}

\begin{rmk}{\rm It is easy to find sufficient conditions on $f$, ensuring that \eqref{e:cd} is satisfied. For instance, if $f$ has the form $f(x) = (2\pi)^{-1/2} q(x) e^{-x^2/2}$, where $q$ is some smooth mapping satisfying $q(x) \geq c >0$ and $|q'(x)|\leq C<\infty$, then one has that $s(x) \leq 1+ \sqrt{2\pi}C/c<\infty$, so that the requirement \eqref{e:cd} is trivially met. 
}
\end{rmk}
\medskip

\noindent{\it Proof of Proposition \ref{p:trio}.} The fact that assumption {\bf (A1)} is satisfied for $g(n) = \sqrt{n}$ is trivial. Moreover, for any choice of $\nn$, the vector $\mathbf{Y}_{\nn}$ appearing at Assumption (i) of Proposition \ref{p:checksm} admits a Stein matrix ${\bf \tau}_{\nn}= \{\tau_{i,j} : i,j=1,...,d\}$ such that, for $i=1,...,d$,
$$
\tau_{i,i} =\frac{1}{n_{2i} - n_{2i-1}}\E\left[ \sum_{k=n_{2i-1}+1}^{n_{2i}} s(Z_k) \, \Big| \, \mathbf{Y}_{\nn}\right],
$$
and $\tau_{i,j} = 0$ for every $i\neq j$. Since ${\bf Var} (\tau_{i,i} )\leq (n_{2i} - n_{2i-1})^{-1} \E[(1 - s(Z_1))^2]$, we deduce immediately that Assumption (ii) of Proposition \ref{p:checksm} is indeed satisfied. To see that relation \eqref{e:cd} implies that Assumption (iii) of Proposition \ref{p:checksm} is also verified, we shall apply Rosenthal inequality for centered random variables (see e.g. \cite[p. 46]{dlpg}), together with the fact that $\E[s(Z_1)] = 1$ and that, for every choice of integers $n_1<n_2$, a Stein factor for $\frac{X_{n_2}-X_{n_1}}{\sqrt{n_2-n_1}}$ is given by
$$
\tau\left( \frac{X_{n_2}-X_{n_1}}{\sqrt{n_2-n_1}}\right) =\frac{1}{n_2-n_1} \sum_{k=n_1+1}^{n_2} s(Z_k).
$$
According to the Rosenthal inequality, one has indeed that, for some universal finite constant $C$ and for every $\theta\geq 2$,
\begin{eqnarray*}
&&\Big\Vert  \tau\left(\frac{X_{n_2}-X_{n_1}}{\sqrt{n_2-n_1}}\right) -1 \Big \Vert_{\theta}\\ 
&&  \leq C \frac{\theta}{\log \theta} \left( \frac{1}{\sqrt{n}} \E[(s(Z_1)-1)^2]^{1/2} + \frac{1}{n^{1-1/\theta}} \E[(s(Z_1)-1)^\theta]^{1/\theta} \right),
\end{eqnarray*}
from which we immediately deduce the desired conclusion. \qed

\medskip

Starting from the next section, we shall focus on sequences of dependent random variables living on a Gaussian space.

}

\subsection{LIL for Gaussian sequences}

As a more substantial application of our main results, we shall now prove a general version of the law of the iterated logarithm for a centered stationary Gaussian sequence $\{Z_k : k\geq 1\}$ with correlation function $r(k):=\E[Z_n Z_{n+k}]$. We write $X_n=\sum_{k=1}^n Z_k.$ 

The following statement is the main result of the section.

\begin{thm}\label{ILLGaussian}
Assume that $g(n):=\sqrt{\E\left[X_n^2\right]}\sim n^a L(n)$, where $a \in (0,1)$ and $L$ is a slowly varying function. Moreover, we assume that 

\begin{equation}\label{covarianceassup}
\sum_{k=n_1}^{n_2} r(k) = O\Big((n_2-n_1)^{2a-1} L(n_2-n_1)\Big).
\end{equation}
Then, we have the following law of the iterated logarithm

\begin{equation}\label{ILLequation}
\limsup_{n\to\infty} \frac{X_n}{\sqrt{2 g^2(n) \log\log n}}=1\,\,\text{a.s.-}\P.
\end{equation}
\end{thm}

\begin{rmk}{\rm
We emphasize that the condition (\ref{covarianceassup}) is strongly related to the fact that $g(n)$ is regularly varying. This can be seen from the equation
$$g^2(n)=2\sum_{k=1}^{n-1}r(k)(n-k)+n.$$
We conjecture that assumption (\ref{covarianceassup}) can indeed be removed, but such an improvement seems difficult for the time being.}
\end{rmk}

The next corollary generalizes some results contained in \cite{arconesGaussian, DEO74, LS1978, Taqqu}: to our knowledge, it corresponds to the most general statement for stationary Gaussian sequences available in the literature. The fact that $\sum_{k} |r(k)|<\infty$ is sufficient to get the law of the iterated logarithm is due to M. Arcones \cite[Corollary 2.1]{arconesGaussian}. Let us stress that Lai and Stout have provided in \cite[Theorem 4]{LS1978} criteria for the law of the iterated logarithm under additional assumptions on the slowly varying function $L$. Case (i) of the next corollary with $b\le 1$ seems to be new.

\begin{cor}\label{CoroGaussian}
Let the assumptions and notation of the present section prevail. The following two implications hold:
\begin{itemize}
\item[(i)] Let $g(n)=n^a L(n)$ with $a \in (0,1)$. If $r(k)=O(\frac{1}{k^b})$ where $b=2-2a$, then (\ref{ILLequation}) holds.
\item[(ii)] If $\sum_{k=1}^\infty |r(k)|<\infty$, then (\ref{ILLequation}) holds.
\end{itemize}
\end{cor}

\subsection{LIL in the Breuer-Major theorem: critical and non-critical regimes}

As anticipated in the Introduction, our main results allow one to deduce sharp laws of the iterated logarithm for the Hermite variations of a fractional Brownian motion. This fact is resumed in the next two propositions. We stress that Theorem \ref{exam-noncritical} can be deduced from \cite[Proposition 1]{arcones}, whereas Theorem \ref{exam-critical} seems to be outside the scope of any other available technique.

\begin{thm}\label{exam-noncritical}
Let $q \ge 2$ and $H_q$ stands for Hermite polynomial of degree $q$. Assume that $B^H=\{ B^H_t\}_{t \in \R}$ be a fractional Brownian motion with Hurst parameter $H < 1 -\frac{1}{2q}$. Set 
$Z_k=B^{H}_{k+1} - B^H_k, \quad k\in \Z$. Define 
\begin{equation}\label{partialsum}
\begin{split}
X_n:=  \sum_{k=0}^{n-1} H_{q}(Z_k) & = \sum_{k=0}^{n-1} H_{q}(B^{H}_{k+1} - B^H_k)\, \quad n \ge1.
\end{split}
\end{equation}
Then, Theorem \ref{mainthm0} with $g(n) \sim \sqrt{n}$ implies that there exists a positive constant $l$ such that 
\begin{equation*}
\limsup_{n\to \infty} \frac{X_n}{\sqrt{2 n \, \log \log n}} = l \quad \text{a.s.-}\P.
\end{equation*}
Here,
$$l=\frac{q!}{2^q}\sum_{r\in\Z}\left(|r+1|^{2H}+|r-1|^{2H}-2|r|^{2H}\right)^q.$$
\end{thm}

\begin{thm}\label{exam-critical}
 Let the notation of Proposition \ref{exam-noncritical} prevail, and set $H=1-\frac{1}{2q}$. Then, applying Theorem \ref{mainthm0} with $g(n) \sim \sqrt{n \log n}$ implies that there exists a positive constant $l$ such 
 that 
\begin{equation*}
\limsup_{n\to \infty} \frac{X_n}{\sqrt{2 n\log n \, \log \log n}}=l   \quad \text{a.s.-}\P.
\end{equation*}
In this case,
$$l=2 q! \left(1-\frac{1}{q}\right)^q \left(1-\frac{1}{2q}\right)^q.$$
\end{thm}

\bigskip

{ The next section contains a number of preliminary results, that will be exploited in the proofs of our main findings.}

\section{Preliminaries}

In this section, we gather together several useful statements, that are needed in order to prove our main results.

\subsection{A result by Lai and Stout}

As anticipated in the Introduction, one of the key contributions of the present paper is a new technique, allowing one to deduce exact lower bounds in the LIL for possibly dependent sequences. For upper bounds, our principal tool will be a classical result by Lai and Stout \cite[Lemma 1]{LS1980}, that we reformulate in a way that is convenient for our discussion.

\begin{lma}[Lemma 1 in  \cite{LS1980}] \label{l:laistout} Let the sequence $X = \{X_n : n\geq 1\}$ verify Assumption {\bf (A1)} (for some appropriate mapping $g$) and Assumption {\bf (A2)}  of Section \ref{ss:genLIL}, and assume that the following two conditions hold:
\begin{itemize}

\item[\rm (a)] For every $0<\epsilon <1$, there exist $\epsilon',K>0$ such that, for $n, a\in \N$ large enough,

$$\P\left(\frac{X_{n+a}-X_a}{g(n)}\ge (1+\epsilon)\sqrt{2\log\log n}\right)\le \frac{K}{\log^{1+\epsilon'} n}.$$

\item[\rm (b)] There exist numbers $\theta,K' >0$  and $B>1$ such that, for $a$ and $n$ large enough,
$$\P\left(\frac{X_{n+a}-X_a}{g(n)}\ge x \sqrt{2\log\log n}\right)\le \frac{K'}{x^{\theta \log\log n} }, \quad x\geq B.$$
\end{itemize}
Then, with $\P$-probability one,
\begin{equation}\label{e:lsupper}
\limsup_{n\to\infty} \frac {|X_n|}{\sqrt{2g(n)^2\log\log n}} \leq1.
\end{equation}

\end{lma}

\begin{rmk}{\rm In view of the assumptions on the mapping $g$ appearing in {\bf (A1)}, one has always that
\begin{equation}\label{e:p1}
\liminf_{n\to\infty} g(K''n)/g(n)>1, \quad \forall K''>1
\end{equation}
and also that, for every $\epsilon>0$, there exists $\rho<1$ such that 
\begin{equation}\label{e:p2}
\limsup_{n\to\infty} \, \left\{\max_{\rho n \leq i\leq n} g(i)/g(n)\right\} \leq  1+\epsilon.
\end{equation}
Relations \eqref{e:p1}--\eqref{e:p2} imply, in particular, that the mapping $n\mapsto \E[X_n^2]$ automatically satisfies relations (1.1)--(1.2) in \cite{LS1980}, that in turn appear as explicit assumptions in the original statement of \cite[Lemma 1]{LS1980}. One should also notice that, in the statement of \cite[Lemma 1]{LS1980}, conditions (a) and (b) require that $K=K'=1$. It is immediately checked that the conclusion remains valid if one considers instead arbitrary finite constants $K,K'>0$.}
\end{rmk}

\subsection{Comparison of multivariate Kolmogorov and 1-Wasserstein distances}
The next result, which is of independent interest, is a crucial step in our approach. We emphasis that the logarithmic dependence on the dimension in the forthcoming estimate \eqref{comparison} is absolutely 
necessary for achieving the proof of our main results.
\begin{thm}\label{t:comparison}
Let $d\ge 1$ be an integer, let $\textbf{X}=(X_1,\cdots,X_d)$ be \textbf{any} random vector and let $\textbf{G}=(G_1,\cdots,G_d)$ be a Gaussian random vector with covariance identity. Then, one has  that
\begin{equation}\label{comparison}
d_{\KK}(\textbf{X},\textbf{G})\le 3 \log^\frac{1}{4}(d+1) \sqrt{W_1(\textbf{X},\textbf{G})}.
\end{equation}
\end{thm}

\begin{proof}
{Without loss of generality, we can assume that ${\bf X}$ and ${\bf G}$ are defined on the same probability space, and also that $\E\big[\|\textbf{X}-\textbf{G}\|\big]=W_1(\textbf{X},\textbf{G})$}. Let $\textbf{t}=(t_1,\cdots,t_d)$. For convenience, we set $\{\textbf{X}\le \textbf{t}\}:=\{X_1\le t_1,\cdots,X_d\le t_d\}$ and $\{\textbf{G}\le \textbf{t}\}=\{G_1\le t_1,\cdots,G_d \le t_d\}$. Let us be given a positive parameter $\epsilon>0$. We have the following inequalities (where we set for simplicity $\|\textbf{x}\|_\infty:=\max_{i=1,\cdots,d}|x_i|$ and $\|\textbf{x}\| :=\|\textbf{x}\|_{\R^d}= \sqrt {x_1^2+\cdots +x_d^2}$):
\begin{eqnarray*}
&&\P(\textbf{X}\le\textbf{t})-\P(\textbf{G}\le\textbf{t})\\
&\le& \P\big(\textbf{X}\le \textbf{t},\|\textbf{X}-\textbf{G}\|_\infty\le\epsilon\big)-\P(\textbf{G}\le\textbf{t})+\frac{1}{\epsilon}\E\big[\|\textbf{X}-\textbf{G}\|_\infty\big]\\
&\le& \Big(\P(\textbf{G}\le \textbf{t}+(\epsilon,\cdots,\epsilon))-\P(\textbf{G}\le \textbf{t})\Big)+\frac{1}{\epsilon}\E\big[\|\textbf{X}-\textbf{G}\|\big]\\
&\le& \Big(\P(\textbf{G}\le \textbf{t}+(\epsilon,\cdots,\epsilon))-\P(\textbf{G}\le \textbf{t})\Big)+\frac{W_1(\textbf{X},\textbf{G})}{\epsilon}
\end{eqnarray*}
In order to estimate the first term, we set
$$\phi(x)=\P\big(\textbf{G}\le \textbf{t}+(x,\cdots,x)\big).$$
One has
$$|\phi(\epsilon)-\phi(0)|\le \sup_{x\in\R}|\phi'(x)|\epsilon.$$
Besides, one has
\begin{eqnarray*}
\phi'(x)&=&\frac{d}{dx}\left(\prod_{i=1}^d\int^{t_i+x}_{-\infty} e^{-\frac{u^2}{2}}\frac{du}{\sqrt{2\pi}}\right)\\
&=&\frac{1}{(2\pi)^\frac{d}{2}}\sum_{i=1}^d e^{-\frac{(t_i+x)^2}{2}}\prod_{j\neq i}\int^{t_j+x}_{-\infty} e^{-\frac{u^2}{2}}du\\
&\le& \theta_d:=\sup_{\textbf{t}\in\R^d} \frac{1}{(2\pi)^\frac{d}{2}}\sum_{i=1}^d e^{-\frac{t_i^2}{2}}\prod_{j\neq i}\int^{t_j}_{-\infty} e^{-\frac{u^2}{2}}du 
\end{eqnarray*}
To estimate $\theta_d$ we follow an iterative scheme. Namely, one has
\begin{eqnarray*}
\theta_d\le \sup_{t_1\in\R}\left(\frac{1}{\sqrt{2\pi}}e^{-\frac{t_1^2}{2}}+\int_{-\infty}^{t_1} e^{-\frac{u^2}{2}}\frac{du}{\sqrt{2\pi}} \theta_{d-1}\right).
\end{eqnarray*}
We are left to estimate the maximum of the next univariate function
$$h(t_1)=\frac{1}{\sqrt{2\pi}}e^{-\frac{t_1^2}{2}}+\int_{-\infty}^{t_1} e^{-\frac{u^2}{2}}\frac{du}{\sqrt{2\pi}} \theta_{d-1}.$$
We have
\begin{eqnarray*}
h'(t_1)=\left(\theta_{d-1}-t_1\right)\frac{e^{-\frac{t_1^2}{2}}}{\sqrt{2\pi}},
\end{eqnarray*}
implying that the maximum of $h$ is reached when $t_1=\theta_{d-1}$. From theses facts, we obtain the following recursion.
$$\theta_d\le \frac{1}{\sqrt{2\pi}}e^{-\frac{\theta_{d-1}^2}{2}}+\theta_{d-1}\int_{-\infty}^{\theta_{d-1}}e^{-\frac{u^2}{2}}\frac{du}{\sqrt{2\pi}}:=f(\theta_{d-1}).$$
We will now show that the previous inequality entails that $\theta_d\le \sqrt{2\log (d+1)}$. We proceed with induction on $d$. When $d=1$, one has $\theta_1\le \frac{1}{\sqrt{2\pi}}\le \sqrt{2\log 2}$. Let us assume now that $d\ge 2$, a straightforward computation implies that $f$ is increasing. Therefore,
\begin{eqnarray*}
\theta_{d+1}&\le& f(\theta_d)\le f(\sqrt{2\log(d+1)})\\
&=&\frac{1}{\sqrt{2\pi}}\frac{1}{d+1}+\sqrt{2 \log(d+1)}\int_{-\infty}^{\sqrt{2\log(d+1)}}e^{-\frac{u^2}{2}}\frac{du}{\sqrt{2\pi}}\\
&\le&\frac{1}{\sqrt{2\pi}}\frac{1}{d+1}+\sqrt{2 \log(d+1)}\Big(1-\frac{1}{(d+1)\sqrt{2\log(d+1)}}\Big)\\
&\le&\sqrt{2\log(d+1)}+(\frac{1}{\sqrt{2\pi}}-1)\frac{1}{d+1}\\
&\le&\sqrt{2\log(d+2)}.\\
\end{eqnarray*}
The same strategy can be implemented to deduce an analogous bound for $\P(\textbf{G}\le\textbf{t})-\P(\textbf{X}\le\textbf{t})$. Putting these facts together, we have showed that, for every $\epsilon>0$,
$$d_{\KK}(\textbf{X},\textbf{G})\le \epsilon \sqrt{2\log(d+1)}+\frac{1}{\epsilon}W_1(\textbf{X},\textbf{G}).$$
A standard argument of optimization implies the desired bound.
\end{proof}

When $d=1$, one recovers from \eqref{comparison} the inequality $d_\KK(X,G) \leq c \sqrt{W_1(X,G)}$, where $c = 3(\log 2)^{1/4} \approx 2.737$. This estimate is slightly worse than the usual bound $d_\KK(X,G) \leq 2 \sqrt{W_1(X,G)}$, see e.g. \cite[formula (C.2.6)]{n-pe-book} and the references therein.

\subsection{Bounds on Wasserstein and Kolmogorov distances in terms of Stein matrices}

{ The following statement shows how Stein's matrices can be directly put into use, in order to asses normal approximations (both in the sense of the Wasserstein and Kolmogorov distances).} Part (a) corresponds to Proposition 3.4 in \cite{HSI}, while Part (b) follows from a standard application of the one-dimensional Stein's method (see e.g. \cite[Chapters 3 and 5]{n-pe-book}).

\begin{prop}\label{p:hsi} Fix an integer $d\geq 1$, as well as $\theta\in [1,\infty)$. Let ${\bf X} = (X_1,..., X_d)$ be any centered random vector whose entries have moments of order $\theta$, and let ${\bf G} = (G_1,..., G_d)$ be a centered standard Gaussian vector. Assume that ${\bf X}$ has a Stein matrix $\tau({\bf X})$ (in the sense of Definition \ref{d:sm}).
\begin{itemize}

\item[\rm (a)] If the entries of $\tau(\bf X)$ have finite moments of order $\theta$, then
\begin{equation}\label{e:wassb}
W_\theta ({\bf X}, {\bf G}) \leq D(d,\theta) \left( \sum_{i,j=1}^d \E\left| \tau_{i,j}({\bf X}) - \delta_{ij} \right|^\theta \right)^{\frac{1}{\theta}},
\end{equation}
where $\delta_{i,j}$ is the Kronecker symbol, and $D(d,\theta) := c_\theta d^{1-1/\theta}$ if $\theta \in [1,2)$, and $D(d,\theta) := c_\theta d^{1-2/\theta}$ if $\theta \geq 2$, with $c_\theta :=  (\E|G_1|^\theta) ^{1/\theta}$. 

\item[\rm (b)] If $d=1$, and therefore ${\bf X} = X$, ${\bf G} = G$ and $\tau(X)$ are one-dimensional random variables,
\begin{equation}\label{e:kb} 
d_{\KK} (X,G) \leq \E\left| \tau(X) - 1\right|.
\end{equation} 
\end{itemize}
\end{prop}

\medskip

{ In order to bound the 1-Wasserstein distance, we will actually need a slightly different bound, proven e.g. by means of Stein's method and of a slight modification of the arguments used in the proof of \cite[Proposition 3.5]{M-S-W}: under the assumptions of Proposition \ref{p:hsi}, and assuming the entries of of $\tau({\bf X})$ are square-integrable,
\begin{equation}\label{e:hg}
W_1({\bf X}, {\bf G}) \leq \sqrt{ \sum_{i,j=1}^d \E[(\tau_{i,j}({\bf X}) - \delta_{ij})^2] }.
\end{equation}
}

\section{Proofs of the main theoretical results}

\subsection{Proof of Theorem \ref{mainthm0}}

\subsubsection{Proof of the upper bound}
Let $g(n)$ be the mapping appearing in Assumption {\bf (A1)}. We shall prove that, under {\bf (A3)}, both Conditions (a) and (b) in the statement of Lemma \ref{l:laistout} are verified, thus implying that the asymptotic upper bound \eqref{e:lsupper} holds with probability one. 

\smallskip

\noindent{\it Verification of Condition} (a). Fix integers $a,n$ such that $2\log\log n> 1$, as well as a real number $p\geq 1$. In view of Assumption {\bf (A3)}, there exists on some auxiliary probability space a coupling $(U,V)$ such that
\begin{eqnarray}
U&\stackrel{\text{law}}{=}&\frac{X_{n+a}-X_a}{g(n)}\notag \\
V&\stackrel{\text{law}}{=} &\mathcal{N}(0,1)\notag\\
\E\left[|U-V|^{2p}\right]&\le& (2p-1)^{p} \left(C \frac{(2p)^\lambda}{1+\log n}\right)^{2p}.\label{e:zaz}
\end{eqnarray}
The Markov inequality yields therefore that, for every $\epsilon\in (0,1)$,
\begin{eqnarray*}
&&\P\left(\frac{X_{n+a}-X_a}{g(n)}\ge (1+\epsilon)\sqrt{2\log\log n}\right)\\
&&\le \P\left(V\ge (1+\frac{\epsilon}{2})\sqrt{2\log\log n}\right)\\
&&+\P\left(|U-V|>\frac{\epsilon}{2}\sqrt{2 \log\log n}\right)\\
&&\le \frac{1}{\log^{(1+\frac{\epsilon}{2})^2} n}+\frac{2^{2p}}{\epsilon^{2p}}\E\left[|U-V|^{2p}\right],\\
\end{eqnarray*}
where we have used the basic estimate $\P[V\geq c]\leq e^{-c^2/2}$, for every $c>1$. Since, the previous bound is valid for any $p$, one can choose $2p=\log\log n$. We now claim that, for $n$ sufficiently large,
$$\frac{2^{\log\log n}}{\epsilon^{\log\log n}} \left(\log\log n\right)^{\log\log n}\left(C \frac{(\log \log n)^\lambda}{1+\log n}\right)^{\log \log n}\le \frac{1}{\log^{(1+\frac{\epsilon}{2}}n)^2} .$$
To see this, just observe that the logarithm of the left hand side of the previous expression is given by
\begin{eqnarray*}
&&\log\log n \times \Big(\log\frac{2}{\epsilon}+\log\log\log n+\\
&&\quad\quad\quad\quad\quad\quad \log C+\lambda \log\log\log n \Big)-\left(\log\log n\right)^2 \sim -\left(\log\log n\right)^2,
\end{eqnarray*}
whereas the logarithm of the right hand side is given by
$$-\left(1+\frac{\epsilon}{2}\right)^2\log\log(n).$$
In view of these relations, we conclude immediately that Condition (a) is verified (for some appropriate $K\geq 1$, by choosing $\epsilon'=(1+\frac{\epsilon}{2})^2-1$ for some $0<\epsilon<2(\sqrt{2}-1)$.
\smallskip

\noindent{\it Verification of Condition} (b). If $n$ is such that $2\log\log n>1$, the same coupling strategy as above yields the bound: for every $x>1$
\begin{eqnarray*}
&&\P\!\left(\!\frac{X_{n+a}-X_a}{g(n)}\!\ge\! x\sqrt{2\log\log n}\right)\!\le\! \frac{1}{\log^{\left( \frac{x}{2}\right)^2} n}\!+\!\frac{2^{2p}}{x^{2p}}\E\left[|U-V|^{2p}\right],
\end{eqnarray*}
where $p\geq 1$ is arbitrary and the coupling $(U,V)$ verifies the bound \eqref{e:zaz}. We now choose $p=2^{-1}\log\log n$, and we shall verify that each of the summands on the right-hand side of the previous inequality is less than $ \frac{1}{x^{ \log\log n} }$ for $n$ large enough. The logarithm of the first summand is $-\left(\frac{x}{2}\right)^2\log\log n$, which is less than $- \log \log n\log x$ for every $x>0$. On the other hand, the logarithm of the second summand is 
\begin{eqnarray*}
&& \log\log n \Big(\log 2+\log\log\log n+\log C\\
&&\quad\quad\quad\quad\quad\quad+\lambda \log\log\log n \Big)-\left(\log\log n\right)^2-\log\log n \log x. 
\end{eqnarray*}
which also verifies the desired inequality, since
$$\log\log n \left(\log 2+\log\log\log n+\log C+\lambda \log\log\log n \right)-\left(\log\log n\right)^2<0.$$
The above computations show that Condition (b) is verified for $B=\theta =1$, and some appropriate $K'\geq 2$. 

 \subsection{Proof of the lower bound}

Let $q>1$ be the real number appearing in Assumption {\bf (A4)}. For any $\epsilon>0$, we select a strictly positive number $\alpha_\epsilon>0$ in such a way that the following condition holds:
\begin{equation}\label{condition}
(1+\alpha_\epsilon)(1-\epsilon)^2<1.
\end{equation}
We need some further notation. Let $d\ge 1$, and let $p=1,2,...$ be an arbitrary integer. We define the $d$-dimensional vector:
\begin{eqnarray*}
&& \textbf{Z}_{p,d}=\left(\frac{X_{q^{{(2p+2)}^{1+\alpha_\epsilon}}}-X_{q^{{(2p+1)}^{1+\alpha_\epsilon}}}}{g\left(q^{{(2p+2)}^{1+\alpha_\epsilon}}-q^{{(2p+1)}^{1+\alpha_\epsilon}}\right)},\right. \\
&&\quad\quad\quad\quad\quad\quad\quad\quad \left.\cdots,\frac{X_{q^{{(2p+2d)}^{1+\alpha_\epsilon}}}-X_{q^{{(2p+2d-1)}^{1+\alpha_\epsilon}}}}{g\left(q^{{(2p+2d)}^{1+\alpha_\epsilon}}-q^{{(2p+2d-1)}^{1+\alpha_\epsilon}}\right)}\right).
\end{eqnarray*}
We also write
\begin{eqnarray*}
&& A_p=\left\{\frac{X_{q^{{(2p+2)}^{1+\alpha_\epsilon}}}-X_{q^{{(2p+1)}^{1+\alpha_\epsilon}}}}{g\left(q^{{(2p+2)}^{1+\alpha_\epsilon}}-q^{{(2p+1)}^{1+\alpha_\epsilon}}\right)}\ge(1-\epsilon)\times \right.\\
&&\quad\quad\quad\quad\quad\quad\quad\quad \left. \times\sqrt{2\log\log\left(q^{{(2p+2)}^{1+\alpha_\epsilon}}-q^{{(2p+1)}^{1+\alpha_\epsilon}}\right)}\right\}.
\end{eqnarray*}
We consider a sequence of i.i.d. standard Gaussian random variables $\{G_i : i\geq 1\}$, and define
$$\textbf{G}_{p,d}=\left(G_p,\cdots,G_{p+d-1}\right).$$
Finally, we introduce the set
$$
B_p=\left\{G_p\ge (1-\epsilon)\sqrt{2\log\log\left(q^{{(2p+2)}^{1+\alpha_\epsilon}}-q^{{(2p+1)}^{1+\alpha_\epsilon}}\right)}\right\}.$$
We shall now prove that $A_p$ {\it is realized infinitely often with $\P$-probability one}. This is indeed the most difficult part of the proof. Indeed, because of lack of independence of the increments of $\{X_n\}_n$, one can not simply use the Borel-Cantelli Lemma. However, the assumption \textbf{(A4)} expresses the fact that, at the particular scale $q^{p^{1+\alpha_\epsilon}}$, the increments become 
sufficiently decorrelated to get the desired result. In order to prove it, we need to translate the amount of information contained in \textbf{(A4)} in terms of Kolmogorov distance between the vector of 
increments and a Gaussian target. This delicate procedure will rely on Theorem \ref{comparison} and Proposition \ref{p:hsi}. We are therefore naturally led to write the following estimates (where 
$C_{q,\epsilon}$ is a constant which only depends on $(q,\epsilon)$ and that may change from line to line):
\begin{eqnarray*}
&& \left|\P\left(\bigcap_{i=p}^{p+d-1} A_i ^c\right)-\P\left(\bigcap_{i=p}^{p+d-1} B_i ^c\right)\right|\\
&& \le d_{\KK}\Big(\textbf{Z}_{p,d},\textbf{G}_{p,d}\Big)\\
&&\le 3\log^\frac{1}{4}(d+1) \sqrt{W_1\Big(\textbf{Z}_{p,d},\textbf{G}_{p,d}\Big)}\quad (\text{by using}\,\,\eqref{comparison})\\
&&\le \frac{C_{q,\epsilon}\log^\frac{1}{4}(d+1)\sqrt{d}}{\sqrt{1+\log\left(q^{{(2p+2)}^{1+\alpha_\epsilon}}-q^{{(2p+1)}^{1+\alpha_\epsilon}}\right)}} \quad(\text{by using}\,\, \textbf{(A4)}).
\end{eqnarray*}
On the other hand, exploiting the independence of the events $B_i$,
\begin{eqnarray*}
\log \P\left(\bigcap_{i=p}^{p+d-1} B_i ^c\right)&=&\sum_{i=p}^{p+d-1}\log \Big(1-\P(B_i)\Big)\\
&\leq&-\sum_{i=p}^{p+d-1}\P(B_i)\\
&\leq&-C_{q,\epsilon}\sum_{i=p}^{p+d-1} \frac{1}{i^{(1+\alpha_\epsilon)(1-\epsilon)^2}}\frac{1}{\sqrt{\log i}}\\
&\leq&-C_{q,\epsilon} \int_p^{p+d-1} \frac{dx}{x^{(1+\alpha_\epsilon)(1-\epsilon)^2}\sqrt{\log x}},
\end{eqnarray*}
where we have used that, if $G\stackrel{\text{law}}{=} \mathcal{N}(0,1)$, then $\P(G>x)>\frac{e^{-{\frac{x^2}{2}}}}{x}.$
Now we choose $\eta>0$ such that $1-\eta>(1+\alpha_\epsilon)(1-\epsilon)^2$. The existence of $\eta$ is indeed supported by the condition \ref{condition}. We have
$$\int_p^{p+d-1} \frac{dx}{x^{(1+\alpha_\epsilon)(1-\epsilon)^2}\sqrt{\log x}}>C_{q,\epsilon}\int_p^{p+d-1} \frac{dx}{x^{1-\eta}}=C_{q,\epsilon}\Big((p+d-1)^\eta-p^\eta\Big).$$
This implies that
$$\log \P\left(\bigcap_{i=p}^{p+d-1} B_i ^c\right)\le -C_{q,\epsilon}\Big((p+d-1)^\eta-p^\eta\Big).$$
Therefore,
\begin{eqnarray*}
\P\left(\bigcap_{i=p}^{p+d-1} A_i ^c\right) &\le& e^{-C_{q,\epsilon}\Big((p+d-1)^\eta-p^\eta\Big)}\\
&&\quad\quad\quad+ \frac{C_{q,\epsilon}\log^\frac{1}{4}(d+1)\sqrt{d}}{\sqrt{1+\log\left(q^{{(2p+2)}^{1+\alpha_\epsilon}}-q^{{(2p+1)}^{1+\alpha_\epsilon}}\right)}}.
\end{eqnarray*}
We can now  take $d=p^x$: if $x>1$ then the first term in the right hand side of the above inequality tends to zero as $p$ tends to infinity. To deal with the second term, we infer that
$$\frac{C_{q,\epsilon}\log^\frac{1}{4}(d+1)\sqrt{d}}{\sqrt{1+\log\left(q^{{(2p+2)}^{1+\alpha_\epsilon}}-q^{{(2p+1)}^{1+\alpha_\epsilon}}\right)}}< C_{q,\epsilon,x} \frac{p^{\frac{x}{2}}}{p^{\frac{1+\alpha_\epsilon}{2}}}\log^{\frac{1}{4}}p.$$
This term goes to zero when $p$ tends to infinity if $x<1+\alpha_\epsilon$. As a matter of fact, for any $1<x<1+\alpha_\epsilon$, we have shown that
$$\lim_{p\to\infty}\P\left(\bigcap_{i=p}^{p+p^x-1} A_i ^c\right)=0.$$
The fact claimed above, namely that $A_p$ is realized infinitely often with probability one, follows at once from the observation that for all $k\ge 1$ 
$$\P\left(\bigcap_{i=k}^{\infty} A_i ^c\right)\le \lim_{p\to\infty}\P\left(\bigcap_{i=p}^{p+p^x-1} A_p ^c\right)=0.$$
We now proceed towards the end of the proof. Recall that we have shown that, almost surely, one has infinitely often that
\begin{eqnarray}\label{i.o.}
&&\frac{X_{q^{{(2p+2)}^{1+\alpha_\epsilon}}}-X_{q^{{(2p+1)}^{1+\alpha_\epsilon}}}}{g\left(q^{{(2p+2)}^{1+\alpha_\epsilon}}-q^{{(2p+1)}^{1+\alpha_\epsilon}}\right)}\\ 
\notag&&\quad\quad\quad \ge (1-\epsilon)\sqrt{2\log\log\left(q^{{(2p+2)}^{1+\alpha_\epsilon}}-q^{{(2p+1)}^{1+\alpha_\epsilon}}\right)}.
\end{eqnarray}
For simplicity, we set $\psi(t)=g(t)\sqrt{2\log\log t}$. First, we will prove that, for any $\alpha>0$, one has that
$$\sum_{k=0}^\infty \P\Big(|X_{q^{k^{1+\alpha}}}|> \psi\big(q^{k^{1+\alpha}}\big)\Big)<\infty.$$
To accomplish this task, we use Assumption $(\textbf{A3})$  to deduce that
\begin{equation}\label{Steinestimate}
 d_{\KK}\left(\frac{X_{q^{k^{1+\alpha}}}}{g(q^{k^{1+\alpha}})},G\right)\le \frac{C}{k^{1+\alpha}}.
\end{equation}
By the triangle inequality and inequality (\ref{Steinestimate}), we get
\begin{equation*}
\sum_{k=1}^\infty \P(A_{k})<C_{q,\alpha} \sum_{k=1}^\infty\frac{1}{k^{1+\alpha}}+\sum_{k=1}^\infty \P\left(G>\sqrt{2 \log\log q^{k^{1+\alpha}}}\right).
\end{equation*}
Hence, since $1+\alpha>1$, the first sum converges. So, one is left to show that the second sum converges as well. Indeed, using the bound $\P(G>t)\le e^{-\frac{t^2}{2}}$ for $t\geq 1$, we have that, for $k$ large enough,
\begin{equation*}
\P\left(G>\sqrt{2 \log\log q^{k^{1+\alpha}}}\right)\le e^{-\log(\log(q^{k^{1+\alpha}}))}\le \frac{C_{q,\alpha}}{k^{1+\alpha}}.
\end{equation*}
By virtue of the first Borel-Cantelli Lemma, we can now infer that, for $p$ large enough, one has
$$
X_{q^{{(2p)}^{1+\alpha_\epsilon}}}\ge - \psi\left(q^{{(2p)}^{1+\alpha_\epsilon}}\right).
$$
Coming back to (\ref{i.o.}), we deduce that almost surely we have infinitely often
$$
X_{q^{{(2p+2)}^{1+\alpha_\epsilon}}}\ge (1-\epsilon) \psi\left(q^{{(2p+2)}^{1+\alpha_\epsilon}}-q^{{(2p+1)}^{1+\alpha_\epsilon}}\right)-\psi\left(q^{{(2p+1)}^{1+\alpha_\epsilon}}\right).
$$
Therefore, almost surely,
\begin{eqnarray*}
&& \limsup_{p\to\infty}\frac{X_{q^{{(2p+2)}^{1+\alpha_\epsilon}}}}{\psi\left(q^{{(2p+2)}^{1+\alpha_\epsilon}}\right)}\\
&&\ge \lim_{p\to\infty} \left((1-\epsilon) \frac{ \psi\left(q^{{(2p+2)}^{1+\alpha_\epsilon}}-q^{{(2p+1)}^{1+\alpha_\epsilon}}\right)}{\psi\left(q^{{(2p+2)}^{1+\alpha_\epsilon}}\right)}-\frac{\psi\left(q^{{(2p+1)}^{1+\alpha_\epsilon}}\right)}{\psi\left(q^{{(2p+2)}^{1+\alpha_\epsilon}}\right)}\right)\\
&&=(1-\epsilon).
\end{eqnarray*}
To obtain the last equality, we have used the fact that 
\begin{eqnarray*}
\lim_{p\to\infty}\frac{\psi\left(q^{{(2p+2)}^{1+\alpha_\epsilon}}-q^{{(2p+1)}^{1+\alpha_\epsilon}}\right)}{\psi\left(q^{{(2p+2)}^{1+\alpha_\epsilon}}\right)}&=&1,\\
\lim_{p\to\infty}\frac{\psi\left(q^{{(2p+1)}^{1+\alpha_\epsilon}}\right)}{\psi\left(q^{{(2p+2)}^{1+\alpha_\epsilon}}\right)}&=&0,
\end{eqnarray*}
which can be easily deduced from the \textit{Karamata integral representation} of the slowly varying function $L$ (see e.g. \cite[p. 14]{Bingbook}).

\subsection{Proof of Proposition \ref{p:checksm}}

We have to check that, under the assumptions in the statement, the three conditions {\bf (A2)}, {\bf (A3)} and {\bf (A4)} are verified.

\smallskip

\noindent{\it Proof of {\bf (A2)}}.  Fix $n,a\in \N$, and let $\tau$ indicate the Stein factor of the random variable $(X_{n+a}- X_a)/g(n)$ (that exists by virtue of (i)). According to the definition of $\tau$, one has that
\begin{eqnarray*}
&&\E\left[\left(\frac{X_{n+a}-X_a}{g(n)}\right)^{2p}\right]\\
&&=(2p-1)\E\left[\left(\frac{X_{n+a}-X_a}{g(n)}\right)^{2p-2}\tau\left(\frac{X_{n+a}-X_a}{g(n)}\right)\right]\\
&&\le(2p-1)\E\left[\left(\frac{X_{n+a}-X_a}{g(n)}\right)^{2p}\right]^\frac{p-1}{p}\E\left[\tau\left(\frac{X_{n+a}-X_a}{g(n)}\right)^p\right]^\frac{1}{p}.
\end{eqnarray*}
Applying Assumption (iii) in the statement, we therefore deduce that
\begin{eqnarray*}
\E\left[\left(\frac{X_{n+a}-X_a}{g(n)}\right)^{2p}\right] &\le&(2p-1)^p\, \E\left[\tau\left(\frac{X_{n+a}-X_a}{g(n)}\right)^p\right].\\
&\le& (2p-1)^p\, \left(1+C\frac{p^\lambda}{g(n)}\right)^p,
\end{eqnarray*}
thus yielding the desired conclusion.

\medskip

\noindent{\it Proof of {\bf (A3)}}. Fix $n_2>n_1$, and let $\tau$ indicate the Stein factor of the random variable $(X_{n_2}- X_{n_1})/g(n_2-n_1)$. According to Proposition \ref{p:hsi}-(a), one has that, for every $\theta\geq 1$
$$
W_\theta \left(\frac{X_{n_2} - X_{n_1}}{g(n_2-n_1)} , G \right) \leq  c_\theta\, \Big\Vert  \tau\big(\frac{X_{n_2}-X_{n_1}}{g(n_2-n_1)}\big) -1 \Big \Vert_{\theta},
$$
so that the desired estimate in the $\theta$-Wasserstein distance follows from \eqref{thetahypercontract}, as well as the bound $c_\theta \leq \alpha(\theta)$. The required one-dimensional bound in the Kolmogorov distance is an immediate consequence of \eqref{thetahypercontract} and \eqref{e:kb}.

\medskip

\noindent{\it Proof of {\bf (A4)}}. {The conclusion follows at once from \eqref{e:hg}, as well as Assumption (ii) in the statement. }


\section{Proofs connected to applications}

In what follows we shall implicitly use the following elementary fact. Let $Z =\{Z_k : k\in \mathbb{Z}\}$ be a centered stationary Gaussian sequence. Then, it is a classical result (use e.g. the results discussed in \cite[Section 2.1]{n-pe-book}) that one can always find an isonormal Gaussian process $G=\{G(h) : h\in \HH\}$ such that the separable Hilbert space $\HH$ contains a sequence $\{h_k : k\in \mathbb{Z}\}$ having the property that $\{G(h_k) : k\in \mathbb{Z}\}$ has the same distribution as $Z$.

\subsection{Proof of Theorem \ref{ILLGaussian}}

We have to check that properties {\bf (A1)}  and (i), (ii) and (iii) in Proposition \ref{p:checksm} are verified. First of all we observe that, since the sequence $X$ is Gaussian, then every vector of the type $Y_{\nn}$ has a Stein matrix given by its own covariance. In view of this fact, it is immediate to check that all the required properties are verified, provided one can show that, for all $j\geq i$,
\begin{equation*}
\left|\frac{\E\left[(X_{n_{2i}}-X_{n_{2i-1}})(X_{n_{2j}}-X_{n_{2j-1}})\right]}{g(n_{2i}-n_{2i-1})g(n_{2j}-n_{2j-1})}\right|\le \frac{C}{1+\log(n_{2i} - n_{2i-1})}.
\end{equation*}
Now, in view of our assumptions, for all $n_{2i-1}\le k \le n_{2i}$,
\begin{eqnarray*}
\left|\sum_{l=n_{2j-1}-k}^{n_{2j}-k}r(l)\right|&\le& C(n_{2j}-n_{2j-1})^{2a-1}L(n_{2j}-n_{2j-1}).
\end{eqnarray*}
and also
\begin{eqnarray*}
&&\left|\frac{\E\left[(X_{n_{2i}}-X_{n_{2i-1}})(X_{n_{2j}}-X_{n_{2j-1}})\right]}{g(n_{2i}-n_{2i-1})g(n_{2j}-n_{2j-1})}\right|\\
&&=\left|\frac{\sum_{k=n_{2i-1}}^{n_{2i}}\sum_{l=n_{2j-1}}^{n_{2j}}r(l-k)}{g(n_{2i}-n_{2i-1})g(n_{2j}-n_{2j-1})}\right|\\
&&\le\frac{n_{2i}-n_{2i-1}}{g(n_{2i}-n_{2i-1})g(n_{2j}-n_{2j-1})}\max_{n_{2i-1}\le k \le n_{2i}}\left|\sum_{l=n_{2j-1}-k}^{n_{2j}-k}r(l)\right|.\\
&&\le\left(\frac{n_{2i}-n_{2i-1}}{n_{2j}-n_{2j-1}}\right)^{1-a}\frac{L(n_{2j}-n_{2j-1})}{L(n_{2i}-n_{2i-1})}.\\
&&\le C_\epsilon\left(\frac{n_{2i}-n_{2i-1}}{n_{2j}-n_{2j-1}}\right)^{1-a-\epsilon},\\
\end{eqnarray*}
where we have used the fact that $\frac{L(n)}{L(m)}\le C_\epsilon\left(\frac{n}{m}\right)^\epsilon$ for any $\epsilon>0$ (see \cite[Theorem 4.4]{LS1978}). Choosing $\epsilon$ small enough leads at once to the desired conclusion.

\subsection{Proofs of Theorem \ref{exam-noncritical} and Theorem \ref{exam-critical}}

For the sake of brevity, we will only focus on the more delicate case of Theorem \ref{exam-critical}, as the non-critical case can be treated in the same way (and is also proved in \cite[Proposition 1]{arcones}). We will check that Assumption  \textbf{(A1)} is verified, together with properties (i), (ii) and (iii) in the statement of Proposition \ref{p:checksm}. We adopt the same notations as in \ref{partialsum}, we set $H=1-\frac{1}{2q}$, and
$$l=\sqrt{2 q! \left(1-\frac{1}{q}\right)^q \left(1-\frac{1}{2q}\right)^q},$$
and we set
$$
g(n):=\sqrt{l n\log(n).}\\
$$
In view of the papers \cite{BM83,DM79,GS85}, it is well known that
$$
\frac{X_n}{g(n)}\xrightarrow[n\to\infty]{\text{law}}~\mathcal{N}(0,1).
$$

\noindent{\it Checking \textbf{(A1)}.} First, by using the stationarity of the increments of a fractional Brownian motion, we infer that
$$X_{n_2}-X_{n_1}\stackrel{\text{Law}}{=}X_{n_2-n_1}.$$
One immediately deduces that
$$\left|\E\Big[  \frac{ X_{n_2}-X_{n_1}}{g(n_2-n_1)} \Big]^2-1\right|=\left|\E\Big[  \frac{ X_{n_2-n_1}}{g(n_2-n_1)} \Big]^2-1\right|.$$
Now we observe that the covariance function $\rho_H$ of the Gaussian sequence 
$\{Z_k\}_{k\in \Z}$ is given by 
$$
\rho_H(k)=\frac 12\big(|k+1|^{2-1/q}-2|k|^{2-1/q}+|k-1|^{2-1/q}\big),
$$ 
and therefore verifies the following straightforward asymptotic relation: 
\begin{equation}
\label{eqDLrho}
\rho_H(k)^q=\left((1-\frac 1{2q})(1-\frac 1q)\right)^q |k|^{-1}+O(|k|^{-3}),\quad\mbox{as $|k|\to\infty$}.
\end{equation}
On the other hand, we have that
\begin{eqnarray*}
&&\E\left[\frac{X_n^2}{g(n)^2}\right]\\
&&=\frac{q!}{l^2 n\log n}\sum_{k,l=0}^{n-1}\rho^q_H(k-l)\\
&&=\frac{1}{\left((1-\frac 1{2q})(1-\frac 1q)\right)^q n\log n}\sum_{k=0}^{n-1}(n-k-1)\rho^q_H(k)\\
&&\!\!\stackrel{(\ref{eqDLrho})}{=}\frac{1}{n\log n}\sum_{k=1}^{n-1}(n-k-1) \frac{1}{k}+O\left(\frac{1}{n\log(n)}\sum_{k=1}^{n-1}(n-k-1)\frac{1}{k^3}\right)\\
&&\quad\quad+O\left(\frac{1}{\log n}\right)\\
&&= I_1+I_2+I_3.
\end{eqnarray*}
First, we notice that
$$I_2\le \frac{1}{\log n}\sum_{k=1}^{n-1}\frac{1}{k^3}=O\left(\frac{1}{\log n}\right).$$
As a consequence, we have only to show that
$$I_1=1+O\left(\frac{1}{\log n}\right).$$
To do so, we use the relations
\begin{eqnarray*}
I_1&=&\frac{1}{n\log n}\left((n-1)\sum_{k=1}^{n-1}\frac{1}{k}-(n-1)\right)\\
&=&\frac{1}{\log n}\sum_{k=1}^{n}\frac{1}{k}+O\left(\frac{1}{\log n}\right)\\
&=&\frac{1}{\log n}\left(\log n +\gamma +O\left(\frac{1}{n}\right)\right)+O\left(\frac{1}{\log n}\right)\\
&=&1+O\left(\frac{1}{\log n}\right),
\end{eqnarray*}
where $\gamma$ stands for the Euler-Mascheroni constant appearing in the asymptotic development of the harmonic series.

\medskip

\noindent{\it Checking {\rm (i)} in Proposition \ref{p:checksm}}.
A consequence of the previous discussion is that $X_n$ can be represented as a sequence of elements of the $q$-th Wiener chaos associated with some isonormal Gaussian process $G = \{G(h) : h\in \HH\}$. The existence of the required Stein matrices follows immediately from relation \eqref{e:simplesm}.

\medskip
\noindent{\it Checking {\rm (ii)} in Proposition \ref{p:checksm}}. Recall once again the explicit expression of the Stein matrix for chaotic random variables given in \eqref{e:simplesm}. Now, in \cite[p. 146]{n-pe-book} it is proved that, writing $\sigma_n^2=\E[X_n^2]$,
\begin{equation}\label{Gamma1isquitegood}
\sqrt{\E\left[\left(1-\frac {1}{q \sigma_n^2} \|D X_n\|_{\HH}^2\right)^2\right]}\le \frac{C}{\log n}.
\end{equation}
By the triangle inequality, we have
\begin{eqnarray*}
&&\sqrt{\E\left[\left(1-\frac {1}{q g(n)^2} \|D X_n\|_{H}^2\right)^2\right]}\\ &&\le
\frac{C}{\log n}+\left|1-\frac{\sigma_n^2}{g(n)^2}\right|\sqrt{\E\left[\left(\frac {1}{q \sigma_n^2} \|D X_n\|_{H}^2\right)^2\right]}.
\end{eqnarray*}
As a consequence, one infers that $\E\left[\left(\frac {1}{q \sigma_n^2} \|D X_n\|_{H}^2\right)^2\right]$ is a bounded sequence by hypercontractivity $(\ref{e:hc})$. Besides, we have showed in checking assumption \textbf{(A1)} that
$$\frac{\sigma_n^2}{g(n)^2}=\frac{1}{g(n)^2}\E[X_n^2]=1+O\left(\frac{1}{\log(n)}\right).$$
It follows that,
\begin{equation}\label{Gamma1isgood}
\sqrt{\E\left[\left(1-\frac {1}{q g(n)^2} \|D X_n\|_{H}^2\right)^2\right]}\le \frac{C}{\log(n)}.
\end{equation}
Making use of the stationarity of the $Z_k$, one can see that 
$$\|D X_n-D X_m\|_H^2\stackrel{\text{Law}}{=}\|D X_{n-m}\|_H^2.$$
This implies that (\ref{varGamma1}) is verified. In order to prove the assumption \textbf{(A4)} (\ref{varcrossGamma1}), we shall use \cite[p 120, Lemma 6.2.1]{n-pe-book}. This Lemma says that for two elements $F,G$ in the same Wiener chaos of order $q$, one has
$$\E\left[<DF,DG>_\HH^2\right]\le C_q\left(\E\left[FG\right]^2+\text{Var}\left[\|DF\|_\HH^2\right]+
\text{Var}\left[\|DG\|_\HH^2\right]\right).$$
We apply such an estimate to
\begin{eqnarray*}
F&=&\frac{X_{n_i}-X_{n_{i-1}}}{g(n_i-n_{i-1})}\\
G&=&\frac{X_{n_j}-X_{n_{j-1}}}{g(n_j-n_{j-1})}.
\end{eqnarray*}
Relying on equation (\ref{Gamma1isgood}), one is left to show that (if $i<j$)
$$ \Big|\E[FG]\Big|\le \frac {C}{1+\log(n_i-n_{i-1})}.$$
\begin{eqnarray*}
&&\Big|\E[FG]\Big|\le \frac{q!}{l^2 g(n_i-n_{i-1})g(n_j-n_{j-1})}\sum_{k=n_{i-1}}^{n_i}\sum_{l=n_{j-1}}^{n_j}|\rho_H(l-k)|^q\\
&&\le \frac{C}{g(n_i-n_{i-1})g(n_j-n_{j-1})}\sum_{k=n_{i-1}}^{n_i}\sum_{l=n_{j-1}}^{n_j}\frac{1}{(l-k)}\\
&&\le\frac{C}{g(n_i-n_{i-1})g(n_j-n_{j-1})}\int_{n_{i-1}}^{n_i}\int_{n_{j-1}}^{n_j}\frac{dx dy}{y-x}\\
&&=\frac{C}{g(n_i-n_{i-1})g(n_j-n_{j-1})}\int_{n_{i-1}}^{n_i}\Big(\log(n_j-x)-\log(n_{j-1}-x)\Big)dx\\
&&\le\frac{C}{g(n_i-n_{i-1})g(n_j-n_{j-1})}\int_{n_{i-1}}^{n_i}\log(n_j-x) dx\\
&&\le\frac{C}{g(n_i-n_{i-1})g(n_j-n_{j-1})}\log(n_j-n_{i-1})(n_i-n_{i-1})\\
&&\le\frac{C \log(n_j-n_{i-1})}{\log(n_i-n_{i-1})\log(n_j-n_{j-1})}\\
&&\le\frac{C}{\log(n_i-n_{i-1})} \frac{\log n_j +\log (1-\frac{n_{i-1}}{n_j})}{\log n_j +\log (1-\frac{n_{j-1}}{n_j})}\\
\end{eqnarray*}
Since when $i,j\to\infty$ we have both $\frac{n_{i-1}}{n_j}\to 0$ and $\frac{n_{j-1}}{n_j}\to 0$ we see that
$$\frac{\log n_j +\log (1-\frac{n_{i-1}}{n_j})}{\log n_j +\log (1-\frac{n_{j-1}}{n_j})}=\frac{1 +\frac{\log (1-\frac{n_{i-1}}{n_j})}{\log n_j}}{1 +\frac{\log (1-\frac{n_{j-1}}{n_j})}{\log n_j}}$$
is a bounded sequence which gives the desired bound.

\medskip

\noindent{\it Checking {\rm (iii)} in Proposition \ref{p:checksm}}. Such an assumption is a straightforward application of (\ref{Gamma1isgood}) and hypercontractivity $(\ref{e:hc})$.

\section*{Acknowledgements}

The authors thank Vincent Munnier for useful discussions to achieve the proof of Theorem \ref{t:comparison}.


\begin{thebibliography}{9}


\bibitem{arconesGaussian}
Arcones, M.A. (1995). On the law of the iterated logarithm for Gaussian processes. \emph{J. Theoret. Probab.} \textbf{8} , no. 4, 877-903. 


\bibitem{arcones}
Arcones, M. A. (1999). The law of the iterated logarithm over a stationary Gaussian sequence of random vectors.
\newblock \emph{J. Theoret. Probab.}, \textbf{12}, no. 3, 615-641.

\bibitem{Beran}
Beran, J. (1994). \newblock \emph{Statistics for long-memory processes}. Monographs on Statistics and Applied Probability, Vol. 61, Chapman and Hall, New York.

\bibitem{Bi11}
Bierm\'{e}, H., Bonami, A., Leon, J. R. (2011). Central limit theorems and quadratic variations in terms of spectral density. \newblock \emph{Electron. J. Probab}, \textbf{16}, no. 13, 362-395.

\bibitem{Bing}
Bingham, N. H. (1986). Variants on the law of the iterated logarithm. \emph{Bull. London Math. Soc.} \textbf{18}, no. 5, 433-467. 

\bibitem{Bingbook}
Bingham, N. H., Goldie, C. M., Teugels, J. L. (1989).\newblock \emph{Regular variation}. \textbf{ 27}, Cambridge university press.

\bibitem{BM83}
Breuer, P., Major, P. (1983). Central limit theorems for nonlinear functionals of Gaussian fields. \emph{J. Multivariate Anal.}, \textbf{13}, no. 3, 425-441.

\bibitem{BN09}
Breton, J.C., Nourdin, I. (2009). Error bounds on the non-normal approximation of Hermite power variations of fractional Brownian motion. \emph{Elec. Com. Prob.} \textbf{13}, 482-493.

\bibitem{Bro09}
Brockwell, P. J., Davis, R. A. (2009). \newblock \emph{Time series: theory and methods}. Springer-Verlag.

\bibitem{cha} S. Chatterjee (2012). A new approach to strong embeddings. \emph{Probab. Theory Related Fields}, {\bf 152}, 231-264.

\bibitem{dlpg} V. H. de la Pe\~{n}na and E. Gin\'e (1999). {\it Decoupling: From Dependence to Independence}. Springer-Verlag.

\bibitem{DEO74}
Deo, C.M.. (1974). A note on stationary gaussian sequences. \emph{Ann. Prob.}. \textbf{2}, no. 5, pp. 954-957.

\bibitem{DM79}
Dobrushin, R. L. , Major, P. (1979). Non-central limit theorems for nonlinear functionals of Gaussian fields. \emph{Z. Wahrsch. verw. Gebiete}, \textbf{50}, pp. 27-52.


\bibitem{Dudley} R.M.\ Dudley (2003). \textit{Real Analysis and Probability }(2$^{\text{nd}}$ Edition). Cambridge University Press, Cambridge.

\bibitem{GS85}
Giraitis, L. and Surgailis, D. (1985). CLT and other limit theorems for functionals of Gaussian processes. \emph{Z. Wahrsch. Verw. Gebiete.} \textbf{70}, no. 2, 191-212.


\bibitem{Ho}
Ho, H. C. (1995). The law of the iterated logarithm for non-instantaneous filters of strongly dependent Gaussian sequences.
  \newblock \emph{J. Theoret. Probab.}, \textbf{8}, no. 2, 347-360.
 
\bibitem{NuTyndel}
Hu, Y.,  Nualart, D., Tindel, S., Xu, F. (2014). Density convergence in the Breuer-Major theorem for Gaussian stationary sequences. \emph{http://arxiv.org/abs/1403.3413}. 
  

   
\bibitem{LS1978}
Lai, T.L., Stout, W. (1978). The law of the iterated logarithm and upper-lower class tests for partial sums of stationary Gaussian sequences. \emph{Ann. Prob.} \textbf{6}, no. 5, 731-750.   
  
  
 
\bibitem{LS1980}
Lai T.L., Stout, W. (1980). Limit Theorems for Sums of Dependent Random variables. \emph{Z. Wahrsch. Verw. Gebiete.}, \textbf{51}, 1-14.

\bibitem{HSI}
Ledoux M., Nourdin I., Peccati G. (2014). Stein's method, logarithmic Sobolev and transport inequalities. \emph{http://arxiv.org/abs/1403.5855}.


\bibitem{nourdin-book}
Nourdin, I. (2012). \newblock \emph{Selected Aspects of Fractional Brownian Motion}. Bocconi and Springer series.

\bibitem{Moda}
Mori, T., Oodaira, H. (1987). The functional iterated logarithm law for stochastic processes represented by multiple Wiener integrals. \emph{Probab. Theory Related Fields}, \textbf{76}, no. 3, 299-310.

\bibitem{M-S-W}
Nourdin I., Peccati G., R\'{e}veillac A. (2010). Multivariate normal approximation using Stein's method and Malliavin calculus. \emph{Ann. Inst. Poincar\'{e}}. \textbf{46}, no. 1, 45-58.


\bibitem{Entropy}
Nourdin, I., Peccati, G., Swan, Y. (2014). Entropy and the fourth moment phenomenon. \emph{ J. Funct. Anal}. \textbf{266}, no. 5, 3170-3207.



\bibitem{NP09}
Nourdin, I., Peccati, G. (2009). Stein's method on Wiener chaos, \emph{Probab. Theory and Related Fields}. \textbf{145}, no. 1-2, 75-118.



\bibitem{n-pe-book} 
Nourdin, I., Peccati, G. (2012).
\newblock \emph{Normal Approximations Using Malliavin Calculus: from Stein's Method to Universality}. 
\newblock Cambridge Tracts in Mathematics. Cambridge University. 

\bibitem{NUbook}
Nualart, D. (2006). \newblock \emph{The Malliavin calculus and related topics}. Springer.
%

\bibitem{S86}
Stein, C. (1986). \emph{Approximate Computation of Expectations}. Institute of Mathematical Statistics Lecture Notes---Monograph Series, 7. Institute of Mathematical Statistics, Hayward, CA.



\bibitem{Stoutbook}
Stout, W., (1974). \newblock \emph{Almost sure convergence}. Probability and Mathematical Statistics, Vol. 24. Academic Press.

\bibitem{Taqqu} Taqqu, M. (1977). Law of the iterated logarithm for sums of non-linear functions of Gaussian variables that exhibit a long range dependence.\emph{Z. Wahrsch. Verw. Gebiete.}, \textbf{40}.  203-238. 

\bibitem{Taqqu2}
 Taqqu, M., Czado, C. (1985). A survey of functional laws of the iterated logarithm for self-similar processes. \emph{Comm. Statist. Stochastic Models}, \textbf{1}, no. 1, 77-115. 


\bibitem{Taqqu3}
Taqqu, M. S. (1975). Weak convergence to fractional Brownian motion and to the Rosenblatt process. \emph{Z. Wahrsch. Verw. Gebiete.}, \textbf{31}, 287-302.






\end{thebibliography}
\end{document}